\documentclass[12pt]{amsart}
\usepackage{amsmath, amssymb, amsthm, mathtools}
\usepackage{comment, amsrefs}
\usepackage{amsaddr}

\usepackage[whole]{bxcjkjatype} 
\usepackage[unicode]{hyperref}

\newcommand{\bZ}{\mathbb{Z}}

\newcommand{\bC}{\mathbb{C}}

\newcommand{\bF}{\mathbb{F}}

\newcommand{\gl}{\mathrm{GL}}
\newcommand{\unitary}{\mathrm{U}}

\newcommand{\cP}{\mathcal{P}}
\newcommand{\cF}{\mathcal{F}}

\DeclareMathOperator{\Irr}{Irr}
\DeclareMathOperator{\cl}{Cl}

\theoremstyle{plain}
\newtheorem{theorem}{Theorem}[section]

\newtheorem{lemma}[theorem]{Lemma}

\newtheorem{conjecture}[theorem]{Conjecture}

\newtheorem*{theorem*}{Theorem}

\numberwithin{equation}{section}

\title[conjecture II for $\gl_n$ and $\unitary_n$]{Harada's conjecture II for the finite general linear groups and unitary groups}

\author{Sugimoto Masahiro}
\date{}
\address{Doctoral Program in Mathematics, Degree Programs in Pure and Applied Sciences, Graduate School of Science and Technology, University of Tsukuba}
\email{sugimoto-m@math.tsukuba.ac.jp}

\begin{document}

\begin{abstract}
  K. Harada conjectured for any finite group $G$, the product of sizes of all conjugacy classes is divisible by the product of degrees of all irreducible characters.
  We study this conjecture when $G$ is the general linear group over a finite field.
  We show the conjecture holds if the order of the field is sufficiently large.
\end{abstract}

\maketitle





\section{Introduction}
Let $G$ be a finite group, and let $\cl(G)=\{K_1,\cdots,K_l\}$ and $\Irr(G)=\{\chi_1,\cdots,\chi_l\}$ denote the set of all conjugacy classes of $G$ and the set of all irreducible characters of $G$ respectively. 
Extending linearly irreducible representations $\rho_i: G\to \gl_{n_i}(\bC)$ of G affording $\chi_i$, we obtain $\bC$-algebra homomorphisms $\tilde{\rho_i}: \bC G\to M_{n_i}(\bC)$ from the group algebra $\bC G$ to the matrix algebra $M_{n_i}(\bC)$.
The restriction of $\bC G\to \prod^l_{i=1} M_{n_i}(\bC)$, obtained from $(\tilde{\rho_i})_{1\leq i\leq l}$, to the center $Z(\bC G)$, is the $\bC$-algebra isomorphism 
$\omega: Z(\bC G) \to Z(\prod^l_{i=1} M_{n_i}(\bC))$.

Let $W$ denote the representation matrix of $\omega$ with respect to the basis $\{\sum_{x \in K_i} x\}_{1\leq i \leq l}$ of $Z(\bC G)$ and the basis $\{(0,\cdots,id_{n_i},\cdots,0)\}_{1\leq i \leq l}$ of $Z(\prod^l_{i=1} M_{n_i}(\bC))= \prod^l_{i=1} Z(M_{n_i}(\bC))$ where $(0,\cdots,id_{n_i},\cdots,0)$ denotes the identity matrix as the $i$-th matrix and the others as zero.
We set $X$ to be the character table of $G$.
$X$ is the representation matrix of $\omega$ with respect to the basis $\{\frac{1}{\# K_i}\sum_{x \in K_i} x\}_{1\leq i \leq l}$ of $Z(\bC G)$ and the basis $\{(0,\cdots,\frac{1}{n_i}id_{n_i},\cdots,0)\}_{1\leq i \leq l}$ of $\prod^l_{i=1} Z(M_{n_i}(\bC))$.
Then, we let 
\[h(G):=\frac{\det W}{\det X}=\prod_{i=1}^{l} \frac{\# K_{i} }{\deg \chi_{i} }.\]

For example, when $G$ is the quaternion group $Q_8$, the character table is 
\begin{table}[htb]
  \centering
  \caption{The character table of $Q_8$}
  \begin{tabular}{c|c|c|c|c|c}
    $Q_8$ & $1$ & $-1$ & $i$ & $j$ & $k$ \\ \hline
    $\chi_1$ & $1$ & $1$ & $1$ & $1$ & $1$ \\
    $\chi_2$ & $1$ & $1$ & $-1$ & $-1$ & $1$ \\
    $\chi_3$ & $1$ & $1$ & $1$ & $-1$ & $-1$ \\
    $\chi_4$ & $1$ & $1$ & $-1$ & $1$ & $-1$ \\
    $\chi_5$ & $2$ & $-2$ & $0$ & $0$ & $0$ \\
  \end{tabular}
\end{table}

Thus we have 
\[h(Q_8)= \frac{1\cdot 1 \cdot 2\cdot 2\cdot 2}{1\cdot 1\cdot 1\cdot 1\cdot 2}= 4.\]

\begin{conjecture}[``Harada's conjecture II''\cite{MR3888878}]\label{conjH}
$h(G)$ is an integer for any finite group $G$.
\end{conjecture}

This question is still open.
Since their character tables are well known, we can easily check abelian groups and dihedral groups satisfy the condition $h(G)\in \bZ$.
Using GAP \cite{gap}, the author verified Conjecture \ref{conjH} for all groups whose orders are less than 1000 excepting $2^9=512$.
A. Hida \cite{hi} proves Conjecture \ref{conjH} for symmetric and alternating groups using the hook length formula.
Some simple groups and their maximal subgroups on ATLAS were confirmed by N. Chigira (unpublished).

Our main results are Theorems \ref{invq} and \ref{largeq} on  the general linear group $\gl_n(k)$ of degree $n$ over the finite field $k$ with $q$ elements.

\begin{theorem} \label{invq}
  Let $k$ be a finite field with $q$ elements.
  Then we have
  \[h(\gl_n(k)) \in \bZ[q^{-1}].\]
\end{theorem}

\begin{theorem} \label{largeq}
  For any fixed $n$, there exists an integer $q_n$ such that if $q>q_n$, then $h(\gl_n(k)) \in \bZ$.
\end{theorem}

Moreover, when $G$ is the finite unitary group, we prove same result (Theorem \ref{uni}) by ``Ennola duality''.
V. Ennola \cite{enn} conjectured certain class functions on the finite unitary group $\unitary_n$ obtained from the irreducible characters of $\gl_n(q)$ is the irreducible characters of $\unitary_n$.
Ennola's conjecture is proved by N. Kawanaka \cite{ka} and called Ennola duality.

\begin{theorem} \label{uni}
  Let $q$ be a power of a prime. Then $h(\unitary_n(q)) \in \bZ[q^{-1}].$
  Moreover, for any fixed $n$, there exists an integer $q_n$ such that if $q>q_n$, then $h(\unitary_n) \in \bZ$.
\end{theorem}

I would like to thank my advisor, Scott Carnahan for his continuous guidance.
This work was supported by JST SPRING, Grant Number JPMJSP2124.

\section{Representations of $\gl_n(k)$}
First, we classify conjugacy classes of $\gl_n(k)$ by characteristic polynomials and partitions.

The order of $\gl_n(k)$ is
\[|\gl_n(k)|=q^{\binom{n}{2}}\psi_n(q)\] 
where $\psi_n(q)=(q-1)(q^2-1)\cdots(q^n-1)$.

For a partition $\lambda= (\lambda_1, \cdots \lambda_l)$ with $\lambda_1 \geq \cdots \geq \lambda_l >0$, we write $l(\lambda)=l$ for the length of $\lambda$ and $\lambda'= (\lambda'_1, \cdots \lambda'_{l'})$ for the conjugate partition of $\lambda$.

Let $f$ be a monic polynomial $f(t) = t^d - a_{d-1}t^{d-1} -\cdots - a_0 \in k[t]$.
We define the $d\times d$ matrix 
\[
U_1(f)=
\begin{pmatrix}
  & 1 & & \\
  & & \ddots & \\
  & & & 1 \\
  a_0 & a_1 & \cdots & a_{d-1} \\
\end{pmatrix}
\]
whose characteristic polynomial is $f(t)$, the $rd \times rd$ matrix
\[
U_r(f)=
\begin{pmatrix}
  U_1(f) & I_d & & \\
  & \ddots & \ddots & \\
  & & & I_d \\
  & & & U_1(f) \\
\end{pmatrix}
\]
where $I_d$ is the unit of $\gl_d(k)$ and, for a partition $\lambda=(\lambda_1, \cdots, \lambda_l)$,
\[
U_{\lambda}(f)=
\begin{pmatrix}
  U_{\lambda_1}(f) & & & \\
  & U_{\lambda_2}(f) & & \\
  & & \ddots & \\
  & & & U_{\lambda_l}(f) \\
\end{pmatrix}.
\]

\begin{lemma}[Green {\cite[Lemma 1.1]{MR72878}}]\label{class}
If the characteristic polynomial $f_{\alpha}$ of $\alpha \in \gl_n(k)$ decomposes as $f_{\alpha}=f_1^{m_1}\cdots f_p^{m_p}$, where $f_1,\cdots,f_p$ are distinct irreducible polynomials over $k$,
then 
$\alpha$ is conjugate to a matrix of the form
\[
\begin{pmatrix}
  U_{\nu_1} & & \\
  & \ddots & \\
  & & U_{\nu_p}
\end{pmatrix}
\]
where $\nu_1,\cdots,\nu_p$ are respective partitions of $m_1,\cdots,m_p$.
\end{lemma}

It follows from Lemma \ref{class} that conjugacy classes of $\gl_n(k)$ are parametrized by maps from the set $\cF$ of monic irreducible polynomials excluding $f(t)=t$ to the set $\cP$ of partitions.

For a map $\nu : \cF \to \cP$, we define
\[
  ||\nu|| = \sum |\nu(f)|\deg(f).
\]

When $||\nu||$ is finite, the number of polynomials $f_1,\cdots, f_p$ which do not map to the empty partition $()$ is finite.

Then if $||\nu|| =n $, the matrices
\[
U_\nu = 
\begin{pmatrix}
  U_{\nu(f_1)} & & \\
  & \ddots & \\
  & & U_{\nu(f_p)}
\end{pmatrix}
\]
are representatives of conjugacy classes.

Let $C_\nu$ be the centralizer of $U_\nu$.
The order of $C_\nu$ is 
\[
  |C_\nu|= \prod_{f\in \cF}a_{\nu(f)}(q^{\deg(f)})
\]
where 
\[
  a_{\lambda}(q)= q^{\binom{l(\lambda)}{2}+\sum_i \lambda'_{i}\lambda'_{i+1}} \prod_i \psi_{\lambda'_i- \lambda'_{i+1}}(q)
\]
for a partition $\lambda$.

Secondly, irreducible characters of $\gl_n(k)$ can be obtained by Green's theory.
Green shows how to obtain the irreducible characters, but we are concerned only with their degree.

\begin{lemma} [Green {\cite[Theorem 14]{MR72878}}]
  There is an explicit map that assigns to each $\nu: \cF \to \cP$ satisfying $|| \nu || = n$ an irreducible character $\chi_\nu$ of $\gl_n(k)$, such that 
  \[\deg \chi_{\nu} = \psi_n(q)\prod_{f\in \cF} b_{\nu(f)}(q^{\deg(f)}) \]
  where
  \[b_\lambda(q)= q^{\sum_i (i-1)\lambda_i} \prod_{i<j} (q^{\lambda_i-\lambda_j-i+j}-1)\prod_{r=1}^{l(\lambda)} \psi_{\lambda_r+l(\lambda)-r}(q)^{-1},\]
   and furthermore, this map is a bijection
\end{lemma}









\section{Proof of main results}

\begin{theorem*}\textbf{\textup{\ref{invq}}}
  Let $k$ be a finite field with $q$ elements.
  Then we have
  \[h(\gl_n(k)) \in \bZ[q^{-1}].\]
\end{theorem*}

\begin{proof}[Proof of Theorem \ref{invq}]

We fix $\nu: \cF \to \cP$ satisfying $||\nu|| = n$. 
By \S 2, the quotient of the size of $K_{\nu}$ and the degree of $\chi_{\nu}$ is
\begin{equation}\label{nu}
  \frac{|\gl_n(k)|}{|C_\nu|\deg \chi_{\nu}}
  =\frac{q^{\binom{n}{2}}}{\prod_{f\in \cF} a_{\nu(f)}(q^{\deg(f)}) b_{\nu(f)}(q^{\deg(f)})}.
\end{equation}

We write $\lambda(r)=\lambda_r -r$ for a partition $\lambda$. Then 
\begin{equation}\label{phi}
  (a_{\lambda}(q)b_{\lambda}(q))^{-1}
  = q^{-\Psi(\lambda)}\frac{\prod \psi_{\lambda(r)+l(\lambda)}(q)}{\prod \psi_{\lambda'_s- \lambda'_{s+1}}(q) \prod_{i<j} (q^{\lambda(i)-\lambda(j)}-1)}
\end{equation}
where 
\[\Psi(\lambda)= \binom{l(\lambda)}{2}+ \sum (i-1)\lambda_i +\sum \lambda'_j \lambda'_{j+1}.\]

It suffices to show that 
\[\Phi(\lambda):=\frac{\prod \psi_{\lambda(r)+l(\lambda)}(q)}{\prod \psi_{\lambda'_s- \lambda'_{s+1}}(q) \prod_{i<j} (q^{\lambda(i)-\lambda(j)}-1)}\]
 is an integer for each partition $\lambda$.
We will prove this using induction on the length of a partition.

\noindent
\textbf{First case}: Length 1

In this case, the conjugate partition is $(1^m)$, so all terms in the first product of the denominator of (\ref{phi}) are $\psi_0(q)=1$.
The second product in the denominator of (\ref{phi}) is empty.
Thus $\Phi(m)=\psi_{m}(q)$ is an integer.

\noindent
\textbf{Second case}: Length at least 2

Let $\lambda$ be a non empty partition and $\lambda^+ = (x, \lambda_1, \cdots, \lambda_l)$ for $x \geq \lambda_1$.
We can easily obtain the following:
\begin{align*}
  \lambda^{+}(i)+l(\lambda^{+}) &= \lambda(i-1)+l(\lambda) &(i > 1)\\
  \lambda^{+}(1)+l(\lambda^{+}) &= x+l(\lambda) \\
  (\lambda^{+})'_i - (\lambda^{+})'_{i+1} &= \lambda'_i- \lambda'_{i+1} &(i<\lambda_1)\\
  (\lambda^{+})'_x - (\lambda^{+})'_{x+1} &= \lambda'_x +1 &(\lambda_1=x)\\
  (\lambda^{+})'_x - (\lambda^{+})'_{x+1} &= 1 &(\lambda_1<x)\\
  \lambda^{+}(i)-\lambda^{+}(j) &= \lambda(i-1)-\lambda(j-1) &(1<i<j)\\
  \lambda^{+}(1)-\lambda^{+}(i) &= x-\lambda(i-1) &(i > 1)
\end{align*}

By these formulas, we have
\begin{align*}
  \Phi(\lambda^+)
  &=\frac{\prod \psi_{\lambda^{+}(r)+l(\lambda^{+})}(q)}{\prod \psi_{(\lambda^{+})'_s - (\lambda^{+})'_{s+1}}(q) \prod_{i<j} (q^{\lambda^{+}(i)-\lambda^{+}(j)}-1)}\\
  &=
  \begin{dcases}
    \Phi(\lambda)\frac{\psi_{x+l(\lambda)}(q)}{(q^{\lambda'_x+1}-1)\prod (q^{x-\lambda(i)}-1)} & (x=\lambda_1) \\
    \Phi(\lambda)\frac{\psi_{x+l(\lambda)}(q)}{\prod (q^{x-\lambda(i)}-1)} & (x>\lambda_1)
  \end{dcases}.
\end{align*}

The integers $x-\lambda(1),\cdots,x-\lambda(l)$ and $\lambda'_x+1$ are distinct and less than or equal to $x+l(\lambda)$.
Indeed, it is clear that 
\[0<x-\lambda(1)<\cdots<x-\lambda(l)<x+l(\lambda).\]
If $x=\lambda_1=\cdots=\lambda_i>\lambda_{i+1}$, then 
\[x-\lambda(i)=i=\lambda'_x<\lambda'_x +1<i+2\leq x-\lambda(i+1),\]
and if $x> \lambda_1$, then 
\[\lambda'_x + 1 = 1 < 1 + (x - \lambda_1) = x-\lambda(1).\]
Therefore $\psi_{x+l(\lambda)}(q)$ is divisible by $(q^{\lambda'_x+1}-1)\prod (q^{x-\lambda(i)}-1)$.
By the induction hypothesis, $\Phi(\lambda^+)$ is an integer.
\end{proof}

It remains to consider the ``$q$-power part'' of $h(\gl_n(k))$.
To simplify notation, we define a valuation $v_q:\bZ[q^{-1}] \to \mathbb{Q} \cup \{\infty\}$;
\[v_q(q^m r)=m \quad (q, r)=1.\]
Then, the value of (\ref{nu}) is 
\[
  \Omega(\nu):=v_q\left(\frac{|\gl_n(k)|}{|C_\nu|\deg \chi_{\nu}}\right)= \binom{n}{2}-\sum_{f\in \cF} \Psi(\nu(f))\deg(f)
\]
and of $h(\gl_n(k))$ is
\[
  v_q(h(\gl_n(k)))=\sum_{||\nu||=n} \Omega(\nu).
\]
By Theorem \ref{invq}, $h(\gl_n(k))\in\bZ$ is equivalent to $v_q(h(\gl_n(k)))\geq 0$.

We rewrite $v_q(h(\gl_n(k)))$ as a polynomial in $q$.
To do this, we introduce an equivalence relation on $\nu$'s.
We define two maps $\nu,\nu' :\cF \to \cP$ to be equivalent if and only if there exists a bijection $\tau$ on $\cF$ preserving degree and satisfying $\nu\circ \tau=\nu'$, and write the equivalence class including $\nu$ as $[\nu]$.
Then $||\nu||=||\nu'||$ and $\Omega(\nu)=\Omega(\nu')$ holds.
Thus the size of $[\nu]$ is at most the number of maps whose value of $\Omega$ equals $\Omega(\nu)$.
By definition, equivalence classes are characterized by the terms in $||\nu||$ satisfying $\nu(f)\neq ()$.

We recall a formula for the number of monic irreducible polynomials (cf. the last equation of \cite[section 1]{MR72878}).
Let $N(d)$ be the number of such polynomials of degree $d$ in $\cF$.
$N(1)=q-1$ and for $d \geq 2$,
\[
  N(d)=\frac{1}{d} \sum_{i|d}\mu\left(\frac{d}{i}\right)q^i,
\]
where $\mu$ is the M\"{o}bius function defined by the following:
\begin{align*}
  \mu(n) =
  \begin{cases}
    1  & n=1 \\
    (-1)^m  & n = \text{the product of}~ m ~\text{different primes}\\
    0 & \text{otherwise}\\
  \end{cases}
\end{align*}

Therefore, $N(d)$ is a polynomial in $q$ of degree $d$.

We can present each class size as a polynomial in $q$, because it is a combination of the number of monic irreducible polynomials.

\textbf{Example.}
Table \ref{n2} gives the data needed to calculate $v_q(h(\gl_2(k)))$.
If $\nu$ maps the polynomial $f(t)=t-1$ to the partition $(2)$ and other polynomials to the empty partition $()$, then
\begin{itemize}
  \item $||\nu||=|(2)|\times 1=2$,
  \item $\Omega(\nu')=\Omega(\nu)=0$ when $\nu$ and $\nu'$ are equivalent, and
  \item the size of $[\nu]$ equals $N(1)=q-1$, the number of degree 1 polynomials in $\cF$.
\end{itemize}
Similarly, we have
\begin{align*}
  \sum_{||\nu||=2} \Omega(\nu) &= 0\times N(1) + (-1) \times N(1) + 1 \times N(2) + 1 \times \binom{N(1)}{2} \\
  &=(q-1)(q-2). 
\end{align*}

\begin{table}[htb]
  \centering
  \caption{$n=2$ ($4$ classes)}
  \begin{tabular}{ccc} \label{n2}
    $[\nu]$ & $\Omega(\nu)$ & $\#[\nu]$ \\ \hline
    $|(2)|\times 1$ & $0$ & $N(1)$ \\ \hline
    $|(1,1)|\times 1$ & $-1$ & $N(1)$ \\ \hline
    $|(1)|\times 2$ & $1$ & $N(2)$ \\ \hline
    $|(1)|\times 1+|(1)|\times 1$ & $1$ & $\binom{N(1)}{2}$ \\
  \end{tabular}
\end{table}

Thus, for $q\geq 2$, $v_q(h(\gl_2(k)))$ is non-negative, so $h(\gl_2(k))$ is an integer.

\begin{lemma}
  $v_q(h(\gl_n(k)))$ is a polynomial in $q$ of degree $n$.
\end{lemma}
\begin{proof}[Proof of Lemma 3.1]
  In the same way as the example, we can show that $\sum \Omega(\nu)$ is a polynomial in $q$ for any $n$.
  For any class $[\nu]$, we have
  \begin{align*}
    \deg \left(\sum_{\nu' \in [\nu]} \Omega(\nu')\right) &= \deg(\#[\nu])\\
    &= \sum_{\nu(f)\neq ()} \deg(N(\deg(f)))\\
    &= \sum_{\nu(f)\neq ()} \deg(f)\\
    &\leq ||\nu||.
  \end{align*}
  The equality holds if and only if $\nu(\cF)= \{(1),()\}$.
  Hence
  \[
    \deg\left(\sum_{||\nu||=n} \Omega(\nu)\right)= n
  \]
\end{proof}

Note that $\Omega(\nu)$ is negative when $||\nu||$ is $|(1,1)|\times 1$.
There are some negative values for general $n$ (See Appendix).
In other words, the ``$q$-power part'' of (\ref{nu}) is not always an integer unlike $\Phi(\lambda)$.
However, if $q$ is large enough, we can ignore negative values.

\begin{theorem*}\textbf{\textup{\ref{largeq}}}
  For any fixed $n$, there exists $q_n$ such that if $q>q_n$, then $h(\gl_n(k)) \in \bZ$.
\end{theorem*}

\begin{proof}[Proof of Theorem \ref{largeq}]
  It suffices to show that the leading coefficient of $\sum \Omega(\nu)$ is positive for any $n$.
  Let $c_{[\nu]}$ be the leading coefficient of the size of $[\nu]$.
  Since $c_{[\nu]}$ is a product of leading coefficients of $N(d)$s,
  the leading coefficient of $\sum \Omega(\nu)$ is
  \[\sum c_{[\nu]}\Omega(\nu)=\binom{||\nu||}{2}\sum c_{[\nu]}>0,\]
  where the sum is over equivalence classes $[\nu]$ satisfying $\nu(\cF)= \{(1),()\}$.
\end{proof}

\section{finite unitary group}
For $\alpha=(a_{ij})\in \gl_n(q^2)$, we write $\alpha^* =(a_{ji}^q)\in \gl_n(q^2)$.
Then the finite unitary group is 
\[\unitary_n(q)= \{\alpha\in \gl_n(q^2) \mid \alpha\alpha^*= I \},\]
where $I$ is the unit of $\gl_n(q^2)$ and the order of $\unitary_n(q)$ is
\[|\unitary_n(q)|=(-1)^n q^{\binom{n}{2}}\psi_n(-q).\]

For a monic polynomial \[f(t)= t^d+a_{d-1}t^{d-1}+\cdots+a_0\] over $\bF_{q^2}$ with $a_0\neq 0$, we denote \[\tilde{f}(t)= \bar{a}_0^{-1}(\bar{a}_0t^d+a_{1}t^{d-1}+\cdots+1).\]
A monic polynomial $f(t)$ is U-irreducible if and only if $f(t)$ is irreducible and $f(t)=\tilde{f}(t)$, or $f(t)=g(t)\tilde{g}(t)$, where $g(t)$ is irreducible and $g(t)\neq\tilde{g}(t)$.
Conjugacy classes of $\unitary_n(k)$ are parametrized by maps from the set $\cF_U$ of monic U-irreducible polynomials excluding $f(t)=t$ to the set $\cP$ of partitions.
By a theorem of Wall \cite{wall}, there exists a bijection from the set of maps $\nu: \cF_U \to \cP$ satisfying $||\nu|| =n$ to the conjugacy classes of $\unitary_n(q)$.

\begin{lemma}
  The size of the conjugacy class $c_\nu$ corresponding to a map $\nu: \cF_U \to \cP$ satisfying $||\nu|| =n$ is
  \[\frac{|\unitary_n(q)|}{(-1)^n a_U(c_\nu)},\]
  where
  \[a_U(c_\nu)= \prod_{f\in \cF_U}a_{\nu(f)}((-q)^{\deg(f)}).\]
\end{lemma}

V. Ennola defines, for each map $\nu: \cF_U \to \cP$, an ``irreducible C-function $\chi_\nu$'' and shows the irreducible C-functions form an orthonormal basis for the vector spabe of class functions on $\unitary_n(q)$ \cite[Theorem 1]{enn}.
N. Kawanaka \cite{ka} proves Ennola's conjecture, that the irreducible C-functions are the irreducible characters of $\unitary_n(q)$.

\begin{lemma}[Ennola duality \cites{enn,ka}]
  The degree of the irreducible character $\chi_{\nu}$ corresponding to a map $\nu: \cF_U \to \cP$ satisfying $||\nu|| =n$ is 
  \[|\psi_n(-q)\prod_{f\in \cF} b_{\nu(f)}((-q)^{\deg(f)})|\]
\end{lemma}

Ennola duality is ``the simple formal change that $q$ is everywhere replaced by $-q$''.
Thus we can apply the proof of theorem \ref{invq} to the finite unitary group.

The number $N_U(d)$ of distinct U-irreducible polynomials of degree $d$ is given in \cite[Theorem 4]{enn}.

\begin{lemma}[Ennola {\cite[Theorem 4]{enn}}]\label{numU}
 \[N_U(d)=\frac{1}{d} (N(d)-c_d),\]
where
\begin{align*}
  c_d =
  \begin{cases}
    -1  & d=1 \\
    2  & d=2\\
    0 & \text{otherwise}\\
  \end{cases}
\end{align*}
\end{lemma}

By Lamma \ref{numU} and the proof of theorem \ref{largeq}, we have theorem \ref{uni}.

\begin{theorem*}\textbf{\textup{\ref{uni}}}
  Let $q$ be a power of a prime. Then $h(\unitary_n(q)) \in \bZ[q^{-1}].$
  Moreover, for any fixed $n$, there exists an integer $q_n$ such that if $q>q_n$, then $h(\unitary_n) \in \bZ$.
\end{theorem*}

\section{Some remarks}
Some calculation results are in the appendix. 
We have calculated $\Omega(\nu)$ and $v_q(h(\gl_n(k)))$ when $n=3,4,5$.
Since $v_q(h(\gl_n(k)))$ is a polynomial in $q$ and its leading coefficient is positive, $h(\gl_n(k))$ is an integer for any $q\geq 2$ if and only if the maximum real root of $v_q(h(\gl_n(k)))$ is less than or equal to $2$.
By Table \ref{vq}, $h(\gl_n(k))\in \bZ$ for $2\leq n \leq 5$.

\begin{table}[htb]
  \centering
  \caption{$v_q(h(\gl_n(k)))$ and the maximum real root}
  \begin{tabular}{ccc}\label{vq}
    $n$ & $v_q(h(\gl_n(k)))$ & max root \\ \hline
    $2$ & $(q-1)(q-2)$ & $2$ \\ \hline
    $3$ & $3(q-1)(q^2-2)$ & $\sqrt{2}$ \\ \hline
    $4$ & $3(q-1)(2q^3+q^2-2q+3)$ & $1$ \\ \hline
    $5$ & $(q-1)(10q^4+7q^3-2q^2-15q-10)$ & $1.174\ldots$ \\
  \end{tabular}
\end{table}

Apparently, $\Omega(\nu)$ is positive for almost all $\nu$ such that $||\nu||=2,3,4$ and $5$.
I think there are few negative values for $n\geq 6$, so I propose the following conjecture.

\begin{conjecture}
  For $n\geq 6$, the maximum real root of $v_q(h(\gl_n(k)))$ is less than $2$.
\end{conjecture}

It seems that when $\Omega (\nu)$ is negative, $\nu$ almost always maps a single degree 1 polynomial to a partition of $n$ and the others to $()$ - see Tables \ref{n3}-\ref{n5}.
Summing $\Omega (\nu)$ over such $\nu$, the partial sum of $v_q(h(\gl_n(k)))$ is
\begin{align*}
  \sum \Omega(\nu) &= \sum \left(\binom{n}{2}-\Psi(\nu(f))\right)\\
  &= (q-1)\left(p(n)\binom{n}{2}-\sum_{\lambda\vdash n}\Psi(\lambda)\right).
\end{align*}

\begin{table}[htb]
  \centering
  \caption{}
  \begin{tabular}{cccc} \label{sumpsi}
    $n$ & $p(n)\binom{n}{2}$ & $\sum_{\lambda\vdash n} \Psi(\lambda)$ \\ \hline
    $2$ & $2$ & $3$ \\ \hline
    $3$ & $9$ & $12$ \\ \hline
    $4$ & $30$ & $36$ \\ \hline
    $5$ & $70$ & $78$ \\ \hline
    $6$ & $165$ & $171$ \\ \hline
    $7$ & $315$ & $309$ \\ \hline
    $8$ & $616$ & $573$ \\ \hline
    $9$ & $1080$ & $960$ \\ \hline
    $10$ & $1890$ & $1611$ \\
  \end{tabular}
\end{table}

From our computations giving Table \ref{sumpsi}, we are led to propose the following conjecture.

\begin{conjecture}
  For $n\geq 7$, $p(n)\binom{n}{2}-\sum_{\lambda\vdash n} \Psi(\lambda)$ is positive.
\end{conjecture}

\section*{Appendix}

\begin{table}[htb]
  \centering
  \caption{$n=3$ ($8$ classes)}
  \begin{tabular}{ccc} \label{n3}
    $[\nu]$ & $\Omega(\nu)$ & $\#[\nu]$ \\ \hline
    $|(3)|\times 1$ & $1$ & $N(1)$ \\ \hline
    $|(2,1)|\times 1$ & $-1$ & $N(1)$ \\ \hline
    $|(1,1,1)|\times 1$ & $-3$ & $N(1)$ \\ \hline
    $|(1)|\times 3$ & $3$ & $N(3)$ \\ \hline
    $|(2)|\times 1+|(1)|\times 1$ & $2$ & $N(1)(N(1)-1)/2$ \\ \hline
    $|(1,1)|\times 1+|(1)|\times 1$  & $1$ & $N(1)(N(1)-1)/2$ \\ \hline
    $|(1)|\times 2+|(1)|\times 1$ & $3$ & $N(2)N(1)$ \\ \hline
    $|(1)|\times 1+|(1)|\times 1+|(1)|\times 1$ & $3$ & $\binom{N(1)}{3}$ \\
  \end{tabular}
\end{table}


\begin{table}[htb]
  \centering
  \caption{$n=4$ ($22$ classes)}
  \begin{tabular}{ccc} \label{n4}
    $[\nu] $ & $\Omega(\nu)$ & $\#[\nu]$ \\ \hline
    $|(4)|\times 1$ & $3$ & $N(1)$ \\ \hline
    $|(3,1)|\times 1$ & $1$ & $N(1)$ \\ \hline
    $|(2,2)|\times 1$ & $-1$ & $N(1)$ \\ \hline
    $|(2,1,1)|\times 1$ & $-3$ & $N(1)$ \\ \hline
    $|(1,1,1,1)|\times 1$ & $-6$ & $N(1)$ \\ \hline
    $|(1)|\times 4$ & $6$ & $N(4)$ \\ \hline
    $|(2)|\times 2$ & $4$ & $N(2)$ \\ \hline
    $|(1,1)|\times 2$ & $2$ & $N(2)$ \\ \hline
    $|(3)|\times 1+|(1)|\times 1$ & $4$ & $N(1)(N(1)-1)/2$ \\ \hline
    $|(2,1)|\times 1+|(1)|\times 1$  & $2$ & $N(1)(N(1)-1)/2$ \\ \hline
    $|(1,1,1)|\times 1+|(1)|\times 1$ & $0$ & $N(1)(N(1)-1)/2$ \\ \hline
    $|(1)|\times 3+|(1)|\times 1$ & $6$ & $N(3)N(1)$ \\ \hline
    $|(2)|\times 1+|(2)|\times 1$ & $4$ & $\binom{N(1)}{2}$ \\ \hline
    $|(2)|\times 1+|(1,1)|\times 1$ & $3$ & $N(1)(N(1)-1)/2$ \\ \hline
    $|(1,1)|\times 1+|(1,1)|\times 1$ & $2$ & $\binom{N(1)}{2}$ \\ \hline
    $|(2)|\times 1+|(1)|\times 2$ & $5$ & $N(1)N(2)$ \\ \hline
    $|(1,1)|\times 1+|(1)|\times 2$ & $4$ & $N(1)N(2)$ \\ \hline
    $|(1)|\times 2+|(1)|\times 2$ & $6$ & $\binom{N(2)}{2}$ \\ \hline
    $|(2)|\times 1+|(1)|\times 1+|(1)|\times 1$ & $5$ & $N(1) \binom{N(1)-1}{2}$ \\ \hline
    $|(1,1)|\times 1+|(1)|\times 1+|(1)|\times 1$ & $4$ & $N(1) \binom{N(1)-1}{2}$ \\ \hline
    $|(1)|\times 2+|(1)|\times 1+|(1)|\times 1$ & $6$ & $N(2)\binom{N(1)}{2}$ \\ \hline
    $|(1)|\times 1+|(1)|\times 1+|(1)|\times 1+|(1)|\times 1$ & $6$ & $\binom{N(1)}{4}$ \\
  \end{tabular}
\end{table}

\begin{table}[htb]
  \centering
  \caption{$n=5$ ($42$ classes)}
  \begin{tabular}{ccc} \label{n5}
    $[\nu] $ & $\Omega(\nu)$ & $\#[\nu]$ \\ \hline
    $|(5)|\times 1$ & $6$ & $N(1)$ \\ \hline
    $|(4,1)|\times 1$ & $4$ & $N(1)$ \\ \hline
    $|(3,2)|\times 1$ & $1$ & $N(1)$ \\ \hline
    $|(3,1,1)|\times 1$ & $0$ & $N(1)$ \\ \hline
    $|(2,2,1)|\times 1$ & $-3$ & $N(1)$ \\ \hline
    $|(2,1,1,1)|\times 1$ & $-6$ & $N(1)$ \\ \hline
    $|(1,1,1,1,1)|\times 1$ & $-10$ & $N(1)$ \\ \hline
    $|(1,1,1,1)|\times 1+|(1)|\times 1$ & $-2$ & $N(1)(N(1)-1)/2$ \\ \hline
    The others & $\geq 0$ & $-$\\
  \end{tabular}
\end{table}

\end{document}